\documentclass[11pt,a4paper,reqno]{amsart}

\usepackage{iftex}

\ifPDFTeX
  \usepackage[T1]{fontenc}
  \usepackage[utf8]{inputenc}
  \usepackage{lmodern}
\else
  \usepackage{fontspec}
  \usepackage{lmodern}
  \defaultfontfeatures{Ligatures=TeX}
\fi

\usepackage[english]{babel}
\usepackage{csquotes}
\usepackage{microtype}

\usepackage[
  a4paper,
  top=1.4in,
  bottom=1.8in,
  left=1.0in,
  right=1.0in
]{geometry}
\usepackage{setspace}
\usepackage{ragged2e}

\usepackage{amsmath,amsthm,amssymb}
\usepackage{mathtools}
\usepackage{bm}
\usepackage{mathrsfs}
\usepackage{esint}
\usepackage{centernot}
\usepackage{cancel}
\usepackage{tensor}

\usepackage{graphicx}
\usepackage{float}
\usepackage{placeins}
\usepackage[font=small,labelfont=bf]{caption}
\usepackage{subcaption}
\usepackage{wrapfig}
\usepackage{rotating}
\usepackage{adjustbox}
\usepackage{xcolor}

\usepackage{tikz}
\usetikzlibrary{
  arrows.meta,
  calc,
  decorations.pathreplacing,
  matrix,
  positioning
}
\usepackage{tikz-cd}

\usepackage{booktabs}
\usepackage{array}
\usepackage{tabularx}
\usepackage{multirow}
\usepackage{longtable}
\usepackage{makecell}
\usepackage{threeparttable}

\usepackage{enumitem}
\usepackage{quoting}
\usepackage{comment}
\usepackage{verbatim}
\usepackage{fancyvrb}
\usepackage{listings}
\usepackage{etoolbox}
\usepackage{xparse}
\usepackage{calc}

\usepackage[
  backend=biber,
  style=numeric,
  sorting=nyt,
  giveninits=true,
  maxbibnames=99,
  url=false,
  doi=false,
  isbn=false,
  backref=false
]{biblatex}
\DeclareBibliographyCategory{bibliographyentriesused}
\newtoggle{bibliographycheckactive}
\newtoggle{bibliographycheckallincluded}

\AtEveryCitekey{%
  \addtocategory{bibliographyentriesused}{\thefield{entrykey}}%
}

\newcommand{\markbibliographyentryused}[1]{%
  \addtocategory{bibliographyentriesused}{#1}%
}
\let\originalnociteforbibliographycheck\nocite
\RenewDocumentCommand{\nocite}{m}{%
  \ifstrequal{#1}{*}
    {\global\toggletrue{bibliographycheckallincluded}}
    {\forcsvlist{\markbibliographyentryused}{#1}}%
  \originalnociteforbibliographycheck{#1}%
}

\AtEveryBibitem{%
  \iftoggle{bibliographycheckactive}{%
    \ifkeyword{bibcheck-ignore}
      {}
      {\PackageWarningNoLine{bibliography-check}{%
         Uncited bibliography entry `\thefield{entrykey}'%
       }}%
  }{}%
}

\makeatletter
\AtEndDocument{%
  \iftoggle{bibliographycheckallincluded}
    {}
    {%
      \ifcsname norefnames\endcsname\norefnames\fi
      \ifcsname setoffmsgs\endcsname\setoffmsgs\fi
      \begin{refsection}
        \originalnociteforbibliographycheck{*}%
        \begingroup
          \toggletrue{bibliographycheckactive}%
          \let\blx@warn@bibempty\relax
          \setbox0=\vbox{%
            \hsize=\textwidth
            \printbibliography[
              notcategory=bibliographyentriesused,
              heading=none
            ]%
          }%
        \endgroup
      \end{refsection}
    }%
}
\makeatother

\usepackage[
  unicode=true,
  colorlinks=true,
  linkcolor=blue,
  citecolor=blue,
  urlcolor=black
]{hyperref}
\usepackage{bookmark}
\usepackage[
  nameinlink,
  noabbrev,
  capitalise
]{cleveref}

\makeatletter
\AtBeginDocument{%
  \let\refcheckoriginalcref\cref
  \RenewDocumentCommand{\cref}{s m}{%
    \forcsvlist{\wrtusdrf}{#2}%
    \IfBooleanTF{#1}
      {\refcheckoriginalcref*{#2}}%
      {\refcheckoriginalcref{#2}}%
  }%
  \let\refcheckoriginalCref\Cref
  \RenewDocumentCommand{\Cref}{s m}{%
    \forcsvlist{\wrtusdrf}{#2}%
    \IfBooleanTF{#1}
      {\refcheckoriginalCref*{#2}}%
      {\refcheckoriginalCref{#2}}%
  }%
  \let\refcheckoriginalcrefrange\crefrange
  \RenewDocumentCommand{\crefrange}{s m m}{%
    \wrtusdrf{#2}\wrtusdrf{#3}%
    \IfBooleanTF{#1}
      {\refcheckoriginalcrefrange*{#2}{#3}}%
      {\refcheckoriginalcrefrange{#2}{#3}}%
  }%
  \let\refcheckoriginalCrefrange\Crefrange
  \RenewDocumentCommand{\Crefrange}{s m m}{%
    \wrtusdrf{#2}\wrtusdrf{#3}%
    \IfBooleanTF{#1}
      {\refcheckoriginalCrefrange*{#2}{#3}}%
      {\refcheckoriginalCrefrange{#2}{#3}}%
  }%
}
\makeatother

\theoremstyle{plain}
\newtheorem{theorem}{Theorem}[section]
\newtheorem{lemma}[theorem]{Lemma}
\newtheorem{proposition}[theorem]{Proposition}

\theoremstyle{definition}
\newtheorem{definition}[theorem]{Definition}

\theoremstyle{remark}

\def\RR{\mathbb{R}}

\def\G{\mathcal{G}}

\def\HH{\mathbb{H}}

\crefname{theorem}{theorem}{theorems}
\crefname{lemma}{lemma}{lemmas}
\crefname{proposition}{proposition}{propositions}
\crefname{corollary}{corollary}{corollaries}
\crefname{conjecture}{conjecture}{conjectures}
\crefname{definition}{definition}{definitions}
\crefname{assumption}{assumption}{assumptions}
\crefname{condition}{condition}{conditions}
\crefname{example}{example}{examples}
\crefname{problem}{problem}{problems}
\crefname{question}{question}{questions}
\crefname{remark}{remark}{remarks}
\crefname{notation}{notation}{notations}
\crefname{observation}{observation}{observations}

\numberwithin{equation}{section}
\allowdisplaybreaks[3]
\patchcmd{\abstract}{\scshape\abstractname}{\bfseries\abstractname}{}{}

\begin{document}

\title[The asymptotic Plateau problem in hyperbolic space]{The asymptotic Plateau problem  for Hypersurfaces of constant $H_{k}$ curvature in hyperbolic space}

\author[X. Mei]{Xinqun Mei}
\address[X. Mei]{Key Laboratory of Pure and Applied Mathematics, School of Mathematical Sciences, Peking University,  Beijing, 100871, P.R. China}
\email{\href{qunmath@pku.edu.cn}{qunmath@pku.edu.cn}}
	
\author[J. Yan]{Jin Yan}
\address[J. Yan]{Institute of Mathematics, Academy of Mathematics and Systems Science, Chinese Academy of Sciences, Beijing 100190, P.R. China}
\email{\href{yanjin@amss.ac.cn}{yanjin@amss.ac.cn}}

\subjclass[2020]{Primary: 53C21. Secondary:35J60, 53C40}

\keywords{The asymptotic Plateau problem, Hypersurfaces of constant curvature, Fully nonlinear elliptic equations.}   

\begin{abstract}
In this paper, we study the asymptotic Plateau problem in hyperbolic space
for hypersurfaces of constant $H_k$-curvature. We prove the existence of a
smooth complete $k$-convex hypersurface in $\mathbb{H}^{n+1}$ satisfying
\[
H_k(\kappa)=\sigma, \qquad \sigma\in(0,1),
\]
with prescribed asymptotic boundary at infinity. In particular, our result extends the range of the constant
$\sigma$ in the existence theorem of Guan and Spruck
[J. Eur. Math. Soc. 12 (2010), no. 3, 797--817] for  $H_{k}$ curvature to the full interval $(0,1)$.
\end{abstract}

\maketitle

\section{Introduction}

Let $\mathbb{H}^{n+1}$ be the $(n+1)$-dimensional hyperbolic space, with
$n\geq 3$, and let $\partial_{\infty}\mathbb{H}^{n+1}$ denote its ideal
boundary at infinity. We use the upper half-space model
\[
\mathbb{H}^{n+1}
=
\left\{
(x,x_{n+1})\in\mathbb{R}^{n+1}:x_{n+1}>0
\right\},
\]
equipped with the hyperbolic metric
\[
\bar{g}
=
\frac{1}{x_{n+1}^{2}}
\sum_{i=1}^{n+1}dx_i^2.
\]
Then $\partial_{\infty}\mathbb{H}^{n+1}$ is naturally identified with
$\mathbb{R}^{n}\times\{0\}\subset \RR^{n+1}$.

In this paper, we are concerned with the asymptotic Plateau problem for
hypersurfaces of constant $H_k$-curvature in $\mathbb{H}^{n+1}$. More
precisely, let $\Omega\subset\mathbb{R}^{n}$ be a bounded smooth domain
and let
\[
\Gamma=\partial\Omega\times\{0\}
\subset\partial_{\infty}\mathbb{H}^{n+1}.
\]
We seek a smooth complete $k$-convex hypersurface
$\Sigma\subset\mathbb{H}^{n+1}$ satisfying
\begin{equation}\label{k-cur}
H_k\bigl(\kappa[\Sigma]\bigr)=\sigma,
\qquad
2\leq k\leq n-1,
\end{equation}
with prescribed asymptotic boundary
\begin{equation}\label{bry}
\partial \Sigma=\Gamma.
\end{equation}
Here
\[
\kappa[\Sigma]
=
(\kappa_1,\ldots,\kappa_n)
\]
denotes the vector of hyperbolic principal curvatures of $\Sigma$, and
$H_k$ is the normalized $k$-th elementary symmetric polynomial,
\[
H_k(\kappa)
\coloneqq
\binom{n}{k}^{-1}\sigma_k(\kappa)
=
\binom{n}{k}^{-1}
\sum_{1\leq i_1<\cdots<i_k\leq n}
\kappa_{i_1}\cdots\kappa_{i_k}.
\]
A hypersurface $\Sigma$ is called $k$-convex if
$\kappa[\Sigma]\in\Gamma_k$, where the G{\aa}rding cone $\Gamma_k$ is
defined by
\[
\Gamma_k
\coloneqq
\left\{
\lambda\in\mathbb{R}^n:
\sigma_j(\lambda)>0
\ \text{for all }1\leq j\leq k
\right\}.
\]

We now state our main theorem.

\begin{theorem}\label{thm:asymptotic-plateau}
Let $\Omega\subset\mathbb{R}^n$ be a bounded smooth domain. Suppose
that $\partial\Omega$ has nonnegative Euclidean mean curvature and that
$\sigma\in(0,1)$. Then, for every $2\leq k\leq n-1$, there exists
\[
u\in C^\infty(\Omega)\cap C^1(\overline{\Omega}),
\qquad
u>0\quad\text{in }\Omega,
\qquad
u=0\quad\text{on }\partial\Omega,
\]
such that its vertical graph
\begin{equation*}
\Sigma
\coloneqq
\bigl\{(x,u(x)):x\in\Omega\bigr\}
\subset\mathbb{H}^{n+1}
\end{equation*}
is a smooth complete $k$-convex hypersurface satisfying
\eqref{k-cur} and \eqref{bry}.
\end{theorem}

The asymptotic Plateau problem in hyperbolic space was initiated by
Anderson \cite{Anderson1982,Anderson1983}, who used methods from
geometric measure theory to construct area-minimizing hypersurfaces
with prescribed asymptotic boundary. The corresponding problem for
constant mean curvature hypersurfaces was subsequently studied by
Tonegawa \cite{Tonegawa1996} and Guan--Spruck
\cite{GuanSpruck2000}. For constant Gauss curvature, the problem was
studied by Labourie \cite{Labourie1991} in $\mathbb{H}^{3}$ and by
Rosenberg--Spruck \cite{RosenbergSpruck1994} in higher dimensions.

For more general symmetric curvature functions, a series of works by
Guan, Spruck, Szapiel, and Xiao
\cite{GuanSpruckSzapiel2009,GuanSpruck2010,GuanSpruckXiao2014}
developed a powerful PDE approach based on representing the desired
hypersurface as a vertical graph and approximating the degenerate
boundary value problem at infinity by nondegenerate Dirichlet
problems. In particular, for a broad class of elliptic curvature
functions, Guan--Spruck--Xiao \cite{GuanSpruckXiao2014} established a
full-range existence theory for locally strictly convex hypersurfaces. The situation is more delicate if one only assumes that the
hypersurface is admissible with respect to a G{\aa}rding cone, rather
than locally strictly convex. In \cite{GuanSpruck2010}, Guan and Spruck
established an existence theorem in this setting under the condition
\[
\sigma\in(\sigma_0^{k},1),
\]
where $\sigma_0\in(0,1)$ is a universal constant satisfying
$
0.3703<\sigma_0<0.3704.
$
Xiao \cite{Xiao2013GD} subsequently enlarged the allowable range of
$\sigma$, lowering the threshold to approximately $0.14$. These works
naturally lead to the following question:

\
\textit{Can the asymptotic Plateau problem of general curvature functions be solved for the admissible hypersurfaces for every
$\sigma\in(0,1)$ ?}

\

Substantial progress on this question has been made in recent years for
the curvature function $H_k^{1/k}(\kappa)$. Lu \cite{Lu2023} solved the case
$k=n-1$. Wang \cite{Wang2025,Wang2026} subsequently established the
full-range existence results for $k=2$ and $k=n-2$, respectively.
Hong--Zhang \cite{HongZhang2024} obtained important curvature estimates
for semi-convex solutions for $2\leq k\leq n-2$. The remaining
existence cases $3\leq k\leq n-3$ are addressed in the present paper.
Consequently, together with the preceding results, the asymptotic
Plateau problem for constant $H_k$-curvature in the $k$-admissible
class is solved for the full range $1\leq k\leq n.
$

We next describe the main idea of the proof. Our argument builds on the
approximation scheme developed by Guan and Spruck
\cite{GuanSpruck2010}. We seek $\Sigma$ as the vertical graph of a
function $u$ over $\Omega$, namely,
\begin{equation}\label{graph-u}
\Sigma
=
\bigl\{
(x,x_{n+1}):x\in\Omega,\ x_{n+1}=u(x)
\bigr\}.
\end{equation}
In this formulation, the geometric problem
\eqref{k-cur}--\eqref{bry} is reduced to the Dirichlet problem for a
fully nonlinear elliptic equation of the form
\begin{equation}\label{g-equ}
G(D^2u,Du,u)=H_k\bigl(\kappa[u]\bigr)=\sigma,
\qquad
u>0
\quad\text{in }\Omega,
\end{equation}
with boundary condition
\begin{equation}\label{u-bry}
u=0
\quad\text{on }\partial\Omega.
\end{equation}
The precise definition of the operator $G$ is given in
\eqref{g-equ-1}; see also \cite[Section~2]{GuanSpruck2010}.

Following the  terminology in \cite{GuanSpruck2010}, a function
$u\in C^2(\Omega)$ is called admissible if
\[
\kappa[u]
\coloneqq
\kappa[\operatorname{graph}(u)]
\in\Gamma_k.
\]
Since equation \eqref{g-equ} degenerates at points where $u=0$, it is
natural to approximate \eqref{u-bry} by the nondegenerate boundary
condition
\begin{equation}\label{ue-bry}
u=\varepsilon>0
\quad\text{on }\partial\Omega.
\end{equation}

An important feature of the work of Guan and Spruck
\cite{GuanSpruck2010} is that, for each sufficiently small fixed
$\varepsilon>0$, the approximating Dirichlet problem
\eqref{g-equ}, \eqref{ue-bry} admits a unique smooth admissible solution
for the full range $\sigma\in(0,1)$. Thus, the restriction on $\sigma$
in their asymptotic existence theorem does not arise from the
solvability of the approximating problems. Rather, it enters in the
global curvature estimate required to pass to the limit
$\varepsilon\to0$. More precisely, the limiting procedure requires a
bound for the hyperbolic principal curvatures of the approximating
hypersurfaces that is uniform in $\varepsilon$. The
maximum-principle argument developed in
\cite[Section~6]{GuanSpruck2010} yields such an estimate under the
condition $\sigma>\sigma_0^{k}$.

This observation identifies the global curvature estimate as the
central issue in extending the $k$-admissible existence theory to the
full range $\sigma\in(0,1)$. For the constant $H_k$-curvature
equation, we establish the required estimate without imposing any
additional restriction on $\sigma$.

\begin{theorem}\label{thm:hyperbolic-global-curvature}
Let $2\leq k\leq n-1$ and $0<\sigma<1$. Let
$u\in C^4(\Omega)\cap C^{2}(\overline{\Omega})$ be an admissible solution of the Dirichlet problem
\eqref{g-equ}, \eqref{ue-bry} satisfying
\begin{eqnarray}\label{omega}
   \vartheta
\coloneqq
\frac{1}{\sqrt{1+|Du|^2}}
\geq\sigma. 
\end{eqnarray}
Let $\Sigma$ be the vertical graph given by \eqref{graph-u}. Then
\begin{equation}\label{eq:hyperbolic-global-curvature}
\max_{\Sigma}|\kappa[u]|
\leq
C\left(
1+\max_{\partial\Sigma}|\kappa[u]|
\right),
\end{equation}
where $C=C(n,k,\sigma)>0$ is independent of $\varepsilon$.
\end{theorem}

The key new ingredient in the proof of
Theorem~\ref{thm:hyperbolic-global-curvature} is a concavity inequality
established by the second author in \cite{yan2026}. Once the global
curvature estimate \eqref{eq:hyperbolic-global-curvature} is available,
the approximation and limiting argument developed by Guan and Spruck
\cite{GuanSpruck2010} can be applied to complete the proof of
Theorem~\ref{thm:asymptotic-plateau}.

\section{Preliminaries}
In this section, we first recall several formulas for  hypersurfaces in hyperbolic space, which can be found in  \cite{GuanSpruckSzapiel2009}. For completeness of this paper, we reproduce them here.  We
 then present some properties of elementary symmetric functions, which will be used in subsequent section.

\subsection{Notation and conventions}
We use the upper half-space model
\[
\mathbb{H}^{n+1}
=
\left\{
(x,x_{n+1})\in\mathbb{R}^{n+1}:x_{n+1}>0
\right\},
\]
equipped with the hyperbolic metric
\[
\bar{g}
=
\frac{1}{x_{n+1}^{2}}
\sum_{i=1}^{n+1}dx_i^2.
\]
Let $\Sigma$ be a hypersurface in $\HH^{n+1}$ and $g$ be the induced metric on $\Sigma$ from $\HH^{n+1}$. 
The Euclidean metric on $\Sigma$ means that the induced metric from $\RR^{n+1}$. Let $X$ denote the position vector of $\Sigma$. The height function of $\Sigma$ in $\mathbb{R}^{n+1}$ is defined by 
\begin{eqnarray*}
    u=\langle X, E_{n+1}\rangle,
\end{eqnarray*}
where $E_{n+1}$ is the standard unit  $(n+1)$-th vector of $\mathbb{R}^{n+1}$ in the $x_{n+1}$-direction and $\langle~, \rangle $ is the Euclidean inner product in $\mathbb{R}^{n+1}$.  Let $\textbf{n}$ and $\nu$ be the unit normal vector field to $\Sigma$ with respect to the hyperbolic metric and Euclidean metric, respectively. Then $\textbf{n}$ and $\nu$ satisfy 
\begin{eqnarray*}
    \textbf{n}= u\nu.
\end{eqnarray*}

We use \(D\) for the Euclidean connection in
\(\mathbb R^{n+1}\), \(\overline\nabla\) for the Levi--Civita connection
of the ambient hyperbolic metric, and \(\nabla\) for the induced
connection on hypersurface $\Sigma$. Throughout this paper, Greek indices range from \(1\) to \(n+1\), while Latin indices range from \(1\) to \(n\). Let $\{e_{i}\}_{i=1}^{n}$ be a local orthonormal frame of vector fields on $(\Sigma, g)$.  The second fundamental form of $\Sigma$ in $\HH^{n+1}$ is locally given by 
\begin{eqnarray*}
    h_{ij}\coloneqq g(\overline{\nabla}_{e_{i}}e_{j}, \textbf{n}),
\end{eqnarray*}
Covariant derivatives on
the hypersurface are indicated by semicolons, that is 
\begin{eqnarray*}
    h_{ij;k}=\nabla_{e_{k}}h_{ij},\quad \quad h_{ij;kl}=\nabla_{e_{l}}\nabla_{e_{k}}h_{ij}.
\end{eqnarray*}

Let $\kappa[\Sigma]\coloneqq (\kappa_{1},\cdots,\kappa_{n})$ and $\tilde\kappa(\Sigma)=(\tilde \kappa_{1},\cdots, \tilde \kappa_{n})$ be the hyperbolic and Euclidean principal curvatures of $\Sigma$， respectively. From \cite[(2.1)]{GuanSpruckSzapiel2009}, we have
\begin{eqnarray}\label{relation}
    \kappa_{i}=u \tilde \kappa_{i}+\nu^{n+1}, \quad 1\leq i\leq n,
\end{eqnarray}
where $\nu^{n+1}=\langle \nu, E_{n+1}\rangle $.

The Codazzi and Gauss eqautions are given by 
equations are
\begin{align}
    h_{ij;k}&=h_{ik;j}, \notag\\
    R_{ijkl}
        &=h_{ik}h_{jl}-h_{il}h_{jk}
          -(g_{ik}g_{jl}-g_{il}g_{jk}).\notag 
        \end{align}
In particular, at some  point $X_{0}\in \Sigma$, choose an orthogonal frame around $X_{0}$, such that 
$\{h_{ij}\}$  is diagonal, then following commutator formula holds
\begin{equation}\label{com}
    h_{pp;qq}= h_{qq;pp}
    +(\kappa_p\kappa_q-1)(\kappa_p-\kappa_q).
\end{equation}

\subsection{Elementary symmetric functions and the curvature operator}
In this subsection, we recall the definition and some basic properties of the elementary symmetric polynomial functions.  
	\begin{definition}
For $k = 1, 2,\ldots, n,$ the $k$-th elementary symmetric function $\sigma_k$ is defined by
\begin{eqnarray*} 
\sigma_k(\lambda) = \sum _{1 \le i_1 < i_2 <\cdots<i_k\leq n}\lambda_{i_1}\lambda_{i_2}\cdots\lambda_{i_k},
 \qquad \text {for} \quad\lambda\coloneqq (\lambda_1,\ldots,\lambda_n)\in \mathbb{R}^{n}.
\end{eqnarray*}
\end{definition}
We use the convention that $\sigma_0=1$ and $\sigma_k =0$ for $k>n$. Let $H_k(\lambda)$ be the normalization of 
$\sigma_{k}(\lambda)$ given by \begin{eqnarray*}
    H_k(\lambda)\coloneqq\frac{1}{\binom{n}{k}}\sigma_k(\lambda).
\end{eqnarray*} 
Denote  $\sigma _k (\lambda \left| i \right.)$ the symmetric
	function with $\lambda_i = 0$ and $\sigma _k (\lambda \left| ij \right.)$ the symmetric function with $\lambda_i =\lambda_j = 0$.  Recall that the  G{\aa}rding cone is defined as
\begin{eqnarray*} 
\Gamma_k \coloneqq \left\{ \lambda  \in \mathbb{R}^n | \sigma _i (\lambda ) > 0,~~\forall 1 \le i \le k \right\}.
\end{eqnarray*} 
\begin{definition}
    Let $W=\{W_{ij}\}$ be an $n\times n$ symmetric matrix. For $k=1,2,\cdots, n$, we define
    \begin{eqnarray*}
        \sigma_{k}(W)\coloneqq \sigma_{k}(\lambda(W))=\sum\limits_{1\leq i_{1}<i_{2}\cdots< i_{k}\leq n}\lambda_{i_{1}}(W)\lambda_{i_{2}}(W)\cdots \lambda_{i_{k}}(W),
    \end{eqnarray*}
where $\lambda(W)\coloneqq (\lambda_1(W),\lambda_2(W),\ldots,\lambda_n(W))$ denotes the vector of eigenvalues of $W$. 
Equivalently, $\sigma_k(W)$ is the sum of all $k\times k$ principal minors of $W$.
\end{definition}
We also denote by $\sigma _m (W\left|
i \right.)$ the symmetric function with $W$ deleting the $i$-row and
$i$-column and $\sigma _m (W \left| ij \right.)$ the symmetric
function with $W$ deleting the $i,j$-rows and $i,j$-columns. 

\begin{proposition}\label{prop2.1}
Suppose that  $W=\{W_{ij}\}$ is a diagonal matrix, and $1\leq k\leq n$. 
Then
\begin{eqnarray*}
\sigma_{k}^{ij}(W)= \begin{cases}
\sigma _{k- 1} W\left| i \right.), &\text{ if } i = j, \\
0, &\text{ if } i \ne j,
\end{cases}
\end{eqnarray*}
where $\sigma_{k-1}^{ij}(W)\coloneqq \frac{{\partial \sigma _k (W)}} {{\partial W_{ij} }}$.
\end{proposition}
\begin{proposition}
For \(\lambda\in\Gamma_k\) and
\(\lambda_1\geq\cdots\geq\lambda_n\), the following property holds:
\begin{equation}\label{eq:garding-one-sided-bound}
    \lambda_i>-\frac{n-k}{k}\lambda_1,
    \qquad 1\leq i\leq n,
\end{equation}
and \begin{eqnarray}\label{1i}
    \sigma_{k-1}(\lambda |1)\leq \sigma_{k-1}(\lambda|i),\quad \forall\quad i\geq 2.
\end{eqnarray}
\end{proposition}

\begin{proof}
See \cite[Lemma~11]{RenWang2023} and \cite[Proposition 2.1 (1)]{wang-in} for the proof of \eqref{eq:garding-one-sided-bound} and \eqref{1i}, respectively.

\end{proof}

Let $\lambda\in \Gamma_{k}$ and $\lambda_{1}>\lambda_{2}\geq \cdots\lambda_{n}$, set
\begin{eqnarray*}
    F(\lambda)\coloneqq \sigma_{k}(\lambda), \quad F^{ii}\coloneqq \sigma_{k-1}(\lambda|i), \qquad F^{ii,jj}
    \coloneqq
    \begin{cases}
        \sigma_{k-2}(\lambda|ij),& i\neq j,\\
        0,& i=j.
    \end{cases}
\end{eqnarray*}
For any
\(\xi=(\xi_1,\ldots,\xi_n)\in\mathbb R^n\) and $\gamma>0$, define
\(\mathcal Q_\gamma(\lambda;\xi)\) by
\begin{eqnarray}\label{Qgamma}
    \mathcal Q_{\gamma}(\lambda;\xi)
    \coloneqq
    \frac{2}{\lambda_{1} F}
    \left(\sum_{i=1}^nF^{ii}\xi_i\right)^2
    -\gamma\frac{F^{11}}{\lambda_{1}^2}\xi_1^2-\frac{1}{\lambda_{1}}\sum_{i,j=1}^nF^{ii,jj}\xi_i\xi_j
    +\frac{2}{\lambda_{1}}\sum_{p=2}^n
    \frac{F^{pp}}{\lambda_{1}-\lambda_p}\,\xi_p^2.
\end{eqnarray}
The following crucial inequality was proved in \cite{yan2026}. 
\begin{theorem}\label{thm-crucial-ineq}
Let \(n\geq3\), \(2\leq k\leq n-1\). Suppose that
\(\lambda\in\Gamma_k\) and 
\(\lambda_1>\lambda_2\geq\cdots\geq\lambda_n\). Assume that 
$$
0<\gamma<\min\{\frac{2k}{n}, 1+\frac{2k-n}{2k^2+n}\}.
$$
Then there exists a small constant $ \eta_*=
    \eta_*(n,k,\gamma)>0$ such that whenever
\begin{eqnarray}\label{eta}
   0<\frac{F}{\lambda_{1}^k}\leq\eta_*,
\end{eqnarray}
we have
\begin{eqnarray*}
    \mathcal Q_{\gamma}(\lambda;\xi)\ge0,
    \qquad
    \forall\xi\in\mathbb R^n.
\end{eqnarray*}
\end{theorem}

\subsection{Formulas for vertical graphs in $\HH^{n+1}$}
Let \(\Omega\subset\mathbb R^n\) be a bounded smooth domain. Suppose that \(u\in C^2(\Omega)\) and $u>0$ in $\Omega$. Consider the vertical graph
\begin{equation*}
  \Sigma \coloneqq
\bigl\{(x,u(x)):x\in\Omega\bigr\}.
\end{equation*}
In the coordinate formulas below, \(u_{x_i}\) and \(u_{x_ix_j}\)
denote ordinary derivatives on \(\Omega\).  Set
\begin{equation*}
    w=\sqrt{1+|Du|^2},
    \qquad
    \nu=\frac{(-Du,1)}{w},
    \qquad
    \vartheta=\nu^{n+1}=\frac{1}{w}.
\end{equation*}
The Euclidean and hyperbolic induced metrics
are
\begin{eqnarray*}
    \tilde{g}_{ij}=\delta_{ij}+u_{x_{i}}u_{x_{j}},\quad \quad g_{ij}=u^{-2}\tilde{g}_{ij},
\end{eqnarray*}
and their inverse matrices  are given by
\begin{eqnarray*}
    \tilde{g}^{-1}_{ij}=\delta_{ij}-\frac{u_{x_{i}}u_{x_{j}}}{\omega^{2}},\quad\quad g^{ij}=u^{2}\tilde{g}^{ij}.
\end{eqnarray*}
The second fundamental form of $\Sigma$ in $\mathbb{R}^{n+1}$ and $\mathbb{H}^{n+1}$ are given by 
\begin{eqnarray*}
    \tilde{h}_{ij}=\frac{u_{x_{i}x_{j}}}{\omega},
\end{eqnarray*}
and 
\begin{eqnarray}\label{hij}
    h_{ij}=\frac{1}{u^{2}\omega}\left(\delta_{ij}+u_{x_{i}}u_{x_{j}}+uu_{x_{i}x_{j}}\right)=\frac{\tilde{h}_{ij}}{u}+\vartheta g_{ij}.
\end{eqnarray}
Then the Euclidean principal curvatures $\tilde \kappa(\Sigma)$ are the eigenvalues of the symmetric matrix
\begin{eqnarray*}
    \tilde{h}_{j}^{i}=\tilde{g}^{ik}\tilde{h}_{kj}\coloneqq \frac{1}{\omega}\gamma^{ik}u_{kl}\gamma^{lj},
\end{eqnarray*}
where 
\begin{eqnarray*}
     \gamma^{ij} =\delta_{ij} -\frac{u_{x_i}u_{x_j}}{w(1+w)}.
\end{eqnarray*}
By \eqref{relation} and \eqref{hij}, the hyperbolic principal curvatures $\kappa[\Sigma]$ are the eigenvalues of the matrix 
\begin{eqnarray*}
    h^{i}_{j}=g^{ik}h_{kj}\coloneqq\frac{1}{\omega}\left(\delta_{ij}+u\gamma^{ik}u_{kl}\gamma^{lj}\right).
\end{eqnarray*}
Then the geometric problem
\eqref{k-cur}--\eqref{bry} is reduced to the Dirichlet problem for a
fully nonlinear elliptic equation
\begin{eqnarray}\label{g-equ-1}
    G(D^{2}u, Du, u)\coloneqq H_{k}(\lambda(h_{j}^{i}))=\sigma \quad \text{in}\quad  \Omega,
\end{eqnarray}
and 
\begin{eqnarray*}
    u=0,\quad \text{on}\quad \partial\Omega.
\end{eqnarray*}
Without causing confusion, we denote 
\begin{eqnarray*}
    F\coloneqq \sigma_{k}(h^{i}_{j}),\quad F^{ij}\coloneqq \frac{\partial F}{\partial h_{j}^{i}},\quad F^{ij,kl}\coloneqq\frac{\partial^{2} F}{\partial h^{i}_{j}\partial h_{l}^{k}}
\end{eqnarray*}

We now switch to intrinsic notation.  Let
\(\{\tau_1,\ldots,\tau_n\}\) be a local hyperbolic orthonormal frame,
and set
\(u_i=\nabla_i u\), \(\vartheta_i=\nabla_i\vartheta\), and
\(\vartheta_{ij}=\nabla_{ij}\vartheta\).  The following identities
are the forms needed in the maximum-principle argument.

\begin{lemma}\label{lem:vertical-normal-identities}
For a vertical graph $\Sigma\in\mathbb{H}^{n+1}$,  suppose that $\Sigma$ satisfies \(F=\sigma_k(h^{i}_{j})=c\) for  some positive constant $c$. Assume at some point $X_{0}\in \Sigma$ and choose \(\{\tau_1,\ldots,\tau_n\}\) be a local orthonormal frame around $X_{0}$, such that
  $h_{i}^{j}(X_{0})=\kappa_i\delta_{ij}.$ Then
\begin{enumerate}
    \item 
    $\sum\limits_{i=1}^n\frac{u_i^2}{u^2}
        =1-\vartheta^2$, \\
   \item $ \vartheta_i
        =\frac{u_i}{u}(\vartheta-\kappa_i),
          \qquad 1\leq i\leq n$,\\
   \item $\sum\limits_{i=1}^nF^{ii}\vartheta_{ii}=2\sum\limits_{i=1}^nF^{ii}\frac{u_i^2}{u^2}
          (\vartheta-\kappa_i)+kF(1+\vartheta^2)
          -\vartheta\sum\limits_{i=1}^nF^{ii}(1+\kappa_i^2)$.
\end{enumerate}
\end{lemma}

\begin{proof}
The first identity follows from Guan--Spruck
\cite[(2.9)]{GuanSpruck2011}. The second identity is
\cite[(2.18)]{GuanSpruck2011}; it is also stated explicitly immediately
after (4.11) therein. For the third identity, apply
\cite[Lemma~4.3, (4.9)]{GuanSpruck2011} to the degree-one curvature
function \(\widehat F=(\sigma_k)^{1/k}\), and multiply the resulting
identity by \(kF^{1-1/k}\), where \(F=\sigma_k\). Since
\[
    \widehat F^{ij}
    =\frac{1}{k}F^{1/k-1}F^{ij},
    \qquad
    \vartheta_i=\frac{u_i}{u}(\vartheta-\kappa_i),
\]
this gives $(3)$.
\end{proof}

\section{Proofs of Theorem \ref{thm:hyperbolic-global-curvature} and Theorem \ref{thm:asymptotic-plateau}}

\begin{proof}[\textbf{Proof of Theorem~\ref{thm:hyperbolic-global-curvature}}.]
Consider the auxiliary function
\begin{eqnarray*}
    \mathcal{G}
    :=
    \log\kappa_{\max}
    -\frac{n+k}{n}\log\vartheta,
\end{eqnarray*}
where $\kappa_{\max}\coloneqq\max\limits_{1\leq i\leq n}\{\kappa_{i}\}$.

Suppose that \(\mathcal{G}\) attains its maximum at some point $x_{0}\in\Sigma$. If $x_{0}\in \partial\Sigma$, then combining \eqref{omega}, we conclude that \eqref{eq:hyperbolic-global-curvature} holds. Next, we assume that $x_{0}\in\Sigma \setminus \partial\Sigma$. 
Choose an orthonormal frame at \(x_0\) such that the second
fundamental form $h_{j}^{i}$ is diagonal, therefore, by Proposition \ref{prop2.1}, we have $F^{ij}$ is also diagonal at $x_0$. Denote
\begin{eqnarray*}
    (h_{ij})
    =
    \operatorname{diag}(\kappa_1,\ldots,\kappa_n),
    \qquad
    a:=\kappa_1=\cdots=\kappa_m
    >\kappa_{m+1}\geq\cdots\geq\kappa_n.
\end{eqnarray*}
By standard approximation argument as in \cite[Corollary 2.3]{yan2026}, we may assume $m=1$. If \(a\) is bounded, there is nothing to prove. We may therefore
assume that \(a>1\) is sufficiently large.

At \(x_0\), the maximum value condition yields
\begin{eqnarray}\label{eq:critical-6.2}
    0
    =
    \G_i
    =
    \frac{h_{11;i}}{a}
    -\frac{n+k}{n}\frac{\vartheta_i}{\vartheta}.
\end{eqnarray}
Combining with the second derivative formula for the largest eigenvalue, see  \cite[Lemma 5]{BCD2017}, we obtain
\begin{eqnarray}
    0
    &\geq&
    \sum_{i=1}^nF^{ii}\G_{ii}
    \nonumber\\
    &\geq&
    \frac1a\sum_{i=1}^nF^{ii}h_{11;ii}
    +\frac2a\sum_{i=1}^n\sum_{p=2}^n
    \frac{F^{ii}h_{1p;i}^2}{a-\kappa_p}
    -\frac1{a^2}\sum_{i=1}^nF^{ii}h_{11;i}^2
    \nonumber\\
    &&
    -\frac{n+k}{n\vartheta}
    \sum_{i=1}^nF^{ii}\vartheta_{ii}
    +\frac{n+k}{n\vartheta^2}
    \sum_{i=1}^nF^{ii}\vartheta_i^2.
    \label{eq:hyperbolic-direct-maximum-inequality-0}
\end{eqnarray}
From \eqref{g-equ-1}, we have
\begin{eqnarray}
    F=\sigma_{k}(\lambda(h^{i}_{j}))=\binom{n}{k}\sigma. \label{F-equ}
\end{eqnarray}
Differentiating Eq. \eqref{F-equ}, we obtain
\begin{eqnarray*}
    \sum_{i=1}^nF^{ii}h_{ii;l}=0,
    \qquad 1\leq l\leq n,
\end{eqnarray*}
and 
\begin{eqnarray}\label{2-der}
    \sum_{i=1}^nF^{ii}h_{ii;11}
    +
    \sum_{p,q,r,s=1}^n
    F^{pq,rs}h_{pq;1}h_{rs;1}
    =
    0.
\end{eqnarray}
By \cite[Theorem 5.1]{ben-2007}, we have
\begin{eqnarray*}
    F^{pq,rs}h_{pq;1}h_{rs;1}=\sum\limits_{i,j=1}^{n}F^{ii,jj}h_{ii;1}h_{jj;1}+2\sum\limits_{k<l}\frac{F^{kk}-F^{ll}}{\kappa_{k}-\kappa_{l}}h_{kl;1}^{2}.
\end{eqnarray*}
Together with \eqref{2-der} and \eqref{1i}, we derive that
\begin{eqnarray*}
\sum_{i=1}^nF^{ii}h_{ii;11}\geq-\sum_{i,j}F^{ii,jj}h_{ii;1}h_{jj;1}+2\sum_{p=2}^n\frac{F^{pp}-F^{11}}{a-\kappa_p}h_{11;p}^2    
\end{eqnarray*}

Recall the commutator formula \eqref{com},
\begin{eqnarray}\label{11}
    h_{11;ii}
    =
    h_{ii;11}
    +(a\kappa_i-1)(a-\kappa_i).
\end{eqnarray}
By Lemma \ref{lem:vertical-normal-identities}, we have
\begin{eqnarray}\label{12}
    \vartheta_i
    =
    \frac{u_i}{u}(\vartheta-\kappa_i)
\end{eqnarray}
and 
\begin{eqnarray}\label{13}
    \sum_{i=1}^nF^{ii}\vartheta_{ii}
    =2\sum_{i=1}^nF^{ii}\frac{u_i^2}{u^2}(\vartheta-\kappa_i)+kF(1+\vartheta^2)-\vartheta\sum_{i=1}^n F^{ii}(1+\kappa_i^2).
\end{eqnarray}

Substituting  \eqref{11}, \eqref{12} and \eqref{13} into
\eqref{eq:hyperbolic-direct-maximum-inequality-0} and collecting the
third-order derivative terms, we obtain
\begin{eqnarray}
    0&\geq&\frac{1}{a}\sum_{i=1}^n F^{ii}(a\kappa_i-1)(a-\kappa_i)+\mathcal Q_{\gamma}+\frac{1}{a^2}\sum_{p=2}^n\frac{a+\kappa_p}{a-\kappa_p}F^{pp}h_{11;p}^2+(\gamma-1)\frac{1}{a^2}F^{11}h_{11;1}^2\nonumber\\
    &&-\frac{k(n+k)}{n}F(\vartheta+\vartheta^{-1})+\frac{n+k}{n}\sum_{i=1}^n F^{ii}(1+\kappa_i^2)+\frac{n+k}{n}\sum_{i=1}^nF^{ii}\frac{u_i^2}{u^2}\left(\frac{\kappa_i^2}{\vartheta^2}-1\right), \label{key-ineu}
\end{eqnarray}
where \(\mathcal Q_\gamma\) is the quadratic form defined
in \eqref{Qgamma}, with
\(\xi_i:=h_{ii;1}\).

Choose 
$$
0<\gamma=\frac{k}{n}<\min\{\frac{2k}{n}, 1+\frac{2k-n}{2k^2+n}\},
$$
We can assume that $a=\lambda_{1}$ is sufficiently large, otherwise, we complete the proof. From \eqref{F-equ}, condition \eqref{eta} is satisfies as $a$ is big enough, and then Theorem \ref{thm-crucial-ineq} yields \begin{eqnarray}\label{g0}
    \mathcal Q_{\gamma}\geq 0.
\end{eqnarray}
On the other hand, by \eqref{eq:critical-6.2}, we have 
\begin{eqnarray}\label{3-1}
   && \frac{1}{a^2}\sum_{p=2}^n\frac{a+\kappa_p}{a-\kappa_p}F^{pp}h_{11;p}^2+(\gamma-1)\frac{1}{a^2}F^{11}h_{11;1}^2 \notag \\
   &=&\frac{(n+k)^{2}}{n^{2}}\left[\sum\limits_{p=2}^{n}\frac{a+\kappa_{p}}{a-\kappa_{p}}F^{pp}\frac{\vartheta_{p}^{2}}{\vartheta^{2}}+(\gamma-1)F^{11} \frac{\vartheta_{1}^{2}}{\vartheta^{2}}\right].
\end{eqnarray}
Inserting \eqref{3-1} and \eqref{g0} into \eqref{key-ineu}, we obtain
\begin{eqnarray*}
0&\geq&kF\left[a+a^{-1}-\frac{n+k}{n}(\vartheta+\vartheta^{-1})\right]+\frac{k}{n}\sum_{i=1}^n F^{ii}(1+\kappa_i^2)\nonumber\\
&&+\frac{n+k}{n}F^{11}\frac{u_1^2}{u^2}\left(1-\frac{a}{\vartheta}\right)\left[\frac{k^2}{n^2}(1-\frac{a}{\vartheta})-2\right]\nonumber\\
&&+\frac{n+k}{n}\sum_{i=2}^n F^{ii}\frac{u_i^2}{u^2}\left(1-\frac{\kappa_i}{\vartheta}\right)\times\left[\left(1+\frac{n+k}{n}\frac{a+\kappa_i}{a-\kappa_i}\right)\left(1-\frac{\kappa_i}{\vartheta}\right)-2\right].\label{eq:hyperbolic-direct-maximum-inequality}
\end{eqnarray*}
Since \(0<\vartheta\leq1<a\), then 
\begin{eqnarray*}
    \frac{n+k}{n}F^{11}\frac{u_1^2}{u^2}\left(1-\frac{a}{\vartheta}\right)\left[\frac{k^2}{n^2}(1-\frac{a}{\vartheta})-2\right]> 0,
\end{eqnarray*}
and 
\begin{eqnarray}
0&\geq&kF\left[a+a^{-1}-\frac{n+k}{n}(\vartheta+\vartheta^{-1})\right]+\frac{k}{n}\sum_{i=1}^n F^{ii}(1+\kappa_i^2)\nonumber\\
&&+\frac{n+k}{n}\sum_{i=2}^n F^{ii}\frac{u_i^2}{u^2}\left(1-\frac{\kappa_i}{\vartheta}\right)\left[\left(1+\frac{n+k}{n}\frac{a+\kappa_i}{a-\kappa_i}\right)\left(1-\frac{\kappa_i}{\vartheta}\right)-2\right].\label{3.11}
\end{eqnarray}

Denote 
\begin{eqnarray*}
    \mathcal{I}_{i}\coloneqq\frac{n+k}{n} F^{ii}\frac{u_i^2}{u^2}\left(1-\frac{\kappa_i}{\vartheta}\right)\left[\left(1+\frac{n+k}{n}\frac{a+\kappa_i}{a-\kappa_i}\right)\left(1-\frac{\kappa_i}{\vartheta}\right)-2\right],\quad 2\leq i\leq n.
\end{eqnarray*}
In order to deal with the term $\{\mathcal{I}\}_{i=2}^{n}$, we divide the indices into three cases. 

\medskip
\noindent
\textbf{Case 1:} $\{i: \kappa_i\geq\vartheta\}$.
In this case,
\begin{eqnarray*}
    1-\frac{\kappa_i}{\vartheta}\leq0,\qquad 1+
    \frac{n+k}{n}
    \frac{a+\kappa_i}{a-\kappa_i}
    >0.
\end{eqnarray*}
It follows that $\mathcal{I}_{i}\geq 0$.

\medskip
\noindent
\textbf{Case 2 :} $\{i: \kappa_i<0\}$.
By \eqref{eq:garding-one-sided-bound}, we have
\begin{eqnarray*}
    \kappa_i>-\frac{n-k}{k}a
\end{eqnarray*}
which implies
\begin{eqnarray}\label{k-1}
    1+
    \frac{n+k}{n}
    \frac{a+\kappa_i}{a-\kappa_i}
    >
    \frac{k(n+2k)}{n^2}
    >0.
\end{eqnarray}
Since \(0<\vartheta\leq1<a\), we have
\begin{eqnarray*}
    1-\frac{\kappa_i}{\vartheta}
    \geq
    1-\kappa_i.
\end{eqnarray*}
A direct calculation yields
\begin{eqnarray}\label{k-2}
    \left(
        1+
        \frac{n+k}{n}
        \frac{a+\kappa_i}{a-\kappa_i}
    \right)
    \left(
        1-\frac{\kappa_i}{a}
    \right)-2
    =
    \frac{k}{n}
    \left(
        1+\frac{\kappa_i}{a}
    \right).
\end{eqnarray}
and 
\begin{eqnarray}\label{k-3}
    1-\frac{\kappa_i}{\vartheta}
    \geq
    1-\kappa_i
    =
    1-\frac{\kappa_i}{a}
    -\kappa_i
    \left(
        1-\frac1a
    \right).
\end{eqnarray}
When $a\geq1+\frac{n}{n+2k}$, combining \eqref{k-1}, \eqref{k-2} and \eqref{k-3}, we obtain
\begin{eqnarray*}
    &&
    \left(
        1+
        \frac{n+k}{n}
        \frac{a+\kappa_i}{a-\kappa_i}
    \right)
    \left(
        1-\frac{\kappa_i}{\vartheta}
    \right)-2
    \\
    &\geq&
    \frac{k}{n}
    \left(
        1+\frac{\kappa_i}{a}
    \right)
    -
    \frac{k(n+2k)}{n^2}
    \kappa_i
    \left(
        1-\frac1a
    \right)
    \\
    &=&
    \frac{k}{n}
    -
    \frac{\kappa_i}{a}
    \left[
        \frac{k(n+2k)}{n^2}(a-1)
        -\frac{k}{n}
    \right]
    \geq
    \frac{k}{n}>0
\end{eqnarray*}
Thus, we  derive that $\mathcal{I}_{i}\geq 0$.

\medskip
\noindent
\textbf{Case 3:} $\{i: 0\leq\kappa_i<\vartheta.\}$
Here
\begin{eqnarray*}
    0<
    1-\frac{\kappa_i}{\vartheta}
    \leq1,
    \qquad
    \frac{a+\kappa_i}{a-\kappa_i}\geq1.
\end{eqnarray*}
Direct calculations yield 
\begin{eqnarray*}
    &&
    \frac{n+k}{n}
    \left(
        1-\frac{\kappa_i}{\vartheta}
    \right)
    \left[
        \left(
            1+
            \frac{n+k}{n}
            \frac{a+\kappa_i}{a-\kappa_i}
        \right)
        \left(
            1-\frac{\kappa_i}{\vartheta}
        \right)-2
    \right]
    \\
    &\geq&
    \frac{n+k}{n}
    \left(
        1-\frac{\kappa_i}{\vartheta}
    \right)
    \left[
        \frac{2n+k}{n}
        \left(
            1-\frac{\kappa_i}{\vartheta}
        \right)-2
    \right]
    \\
    &=&
    \frac{(n+k)(2n+k)}{n^2}
    \left(
        1-\frac{\kappa_i}{\vartheta}
        -\frac{n}{2n+k}
    \right)^2
    -\frac{n+k}{2n+k}
    \\
    &\geq&
    -\frac{n+k}{2n+k}.
\end{eqnarray*}

Since \(\kappa_i\geq0\), we have
\begin{eqnarray*}
    F^{ii}
    \leq
    \sigma_{k-1}(\kappa)
    =
    \frac{1}{n-k+1}
    \sum_{j=1}^nF^{jj},
\end{eqnarray*}
and 
\begin{eqnarray*}
    \sum_{i=1}^n\frac{u_i^2}{u^2}
    =
    1-\vartheta^2
    \leq1.
\end{eqnarray*}
We conclude that
\begin{eqnarray*}
    \mathcal{I}_{i}\geq   -\frac{n+k}{(2n+k)(n-k+1)}
    \sum_{j=1}^nF^{jj}.
\end{eqnarray*}
By taking  account into Case 1, Case 2, and Case 3, we derive that
\begin{eqnarray*}
    \sum\limits_{i=2}^{n}\mathcal{I}_{i}\geq  -\frac{n+k}{(2n+k)(n-k+1)}
    \sum_{j=1}^nF^{jj}.
\end{eqnarray*}
Combining this estimate with
\eqref{3.11}, we obtain
\begin{eqnarray*}
    0
    &\geq&
    kF\left[
        a+\frac1a
        -\frac{n+k}{n}
        \left(
            \vartheta+\frac1\vartheta
        \right)
    \right]
    +\frac{k}{n}
    \sum_{i=1}^nF^{ii}\kappa_i^2
    +
    \left[
        \frac{k}{n}
        -
        \frac{n+k}{(2n+k)(n-k+1)}
    \right]
    \sum_{i=1}^nF^{ii}\\
    &\geq &kF\left[
        a+\frac1a
        -\frac{n+k}{n}
        \left(
            \vartheta+\frac1\vartheta
        \right)
    \right].
\end{eqnarray*}

Since \(\sigma\leq\vartheta\leq1\), we have
\begin{eqnarray*}
    \vartheta+\frac 1\vartheta
    \leq
    1+\frac1\sigma.
\end{eqnarray*}
Hence
\begin{eqnarray*}
    a\leq C,
\end{eqnarray*}
where the positive constant $C$ depends on $n,k,\sigma$. 
We complete the proof of  \eqref{eq:hyperbolic-global-curvature}.
\end{proof}

Finall, we complete the proof of Theorem \ref{thm:asymptotic-plateau}.
\begin{proof}[\textbf{Proof of Theorem~\ref{thm:asymptotic-plateau}.}]
As pointed out in \cite{GuanSpruck2010}, it only remains to establish global-to-boundary curvature estimates. In view of Theorem \eqref{thm:hyperbolic-global-curvature}, we apply the approximation and compactness argument developed in \cite{GuanSpruck2010} to establish Theorem \ref{thm:asymptotic-plateau}.
\end{proof}

\bigskip

\bigskip

\noindent\textit{Acknowledgment:} X.M. was supported by the National Key R $\&$ D Program of China (No. 2020YFA0712800) and the Postdoctoral Fellowship Program of CPSF under Grant Number 2025T180843 and 2025M773082.

\printbibliography

\end{document}